\documentclass{article}
\usepackage[cm]{fullpage}
\usepackage[T1]{fontenc}
\usepackage{amsmath}
\usepackage{amssymb}
\usepackage{amsthm}
\usepackage{titlesec}
\usepackage{enumerate}
\usepackage{mathabx}
\usepackage{mathrsfs}
\usepackage{multicol}

\newcommand{\refp}[1]{~(\ref{#1})}

\newtheorem{theorem}{Theorem}

\allowdisplaybreaks[4]

\title{First-passage duality in $d>2$ dimensions}
\author{Merek Johnson}
\date{\today}

\begin{document}

\twocolumn[
  \begin{@twocolumnfalse}
    \maketitle
    \begin{abstract}
      In one and two dimensions, the first-passage time for a diffusing particle in the presence of a radial potential flow to hit a sphere, conditioned on actually hitting the sphere, is independent of the sign of the drift.  Moreover, the first-passage distribution for an inward (negative) drift is identical to the first-passage distribution for outward (positive) drift when conditioned on hitting the sphere.  This curious property was first observed in 2018 by Krapivsky and Redner, and in this work we extend this duality of first-passage times and their distributions to dimensions $d>2$.  We show that the fundamental source of the symmetry is the fact that a reversal of the direction of the drift amounts to taking the adjoint of the spatial operator governing the evolution of the process, and the factor that converts between solutions of the forward and backward equations is precisely the non-transient behavior of the system with which to appropriately condition by.  \\ \vspace{0.5cm}
    \end{abstract}
  \end{@twocolumnfalse}
]


\section{Introduction}

The study of first-passage times is motivated by the guiding question: How long might one expect it to take for a diffusing particle to reach some target?  This question naturally arises in a broad range of fields, including in the study of chemical processes~\cite{Gillespie2}\cite{Gillespie1}, protein folding behavior~\cite{proteins}, stock options pricing~\cite{Black}, and extinction rates of infectious disease~\cite{extinct}.  First-passage times and more generally, hitting times, have been studied extensively.  Nonetheless, even some of the simplest low-dimensional processes exhibit sufficiently rich dynamics that curious properties of first-passage time behavior continue to emerge.  In 2018, Krapivsky and Redner~\cite{duality} first noted a rather remarkable duality for one and two dimensional diffusions in the presence of a radial potential flow.  They showed that the mean first-passage time for a particle with initial position $r$ to hit a target point $a<r$ or circle of radius $a<|r|$, conditioned on \textit{actually hitting the target}, is independent of the sign of the drift.  In fact, they show that not only are the hitting-times equal but the first-passage distributions themselves are identical when appropriately conditioned.  

In this paper, we extend this duality to dimensions $d>2$ where the target set is a $d$-sphere and the drift is characterized by a radially symmetric potential.  In particular, we show that the first-passage distribution under a negative (inward) drift is precisely equal to the first-passage distribution under a positive (outward) drift when conditioned on hitting the hypersphere.  Furthermore, we establish that the underlying mathematical symmetry is due to the fact that changing the sign of the drift is equivalent to taking the adjoint of the spatial operator governing the process, and therefore amounts to converting between the forward and backward Kolmogorov equations.  The factor that allows us to convert between solutions to the forward and backward equations is precisely the non-transient behavior of the system that appropriately conditions the first-passage distribution.  

In Section 2 we describe the setting in $d=1$ and $d=2$ and make precise the duality as noted by Krapivsky and Redner~\cite{duality}.  We adopt much of their notation and methods for the sake of exposition and to ensure that this work be somewhat self-contained.  In Section 3 we establish the appropriate higher-dimensional setting and prove the duality in $d>2$.  In Section 4 we consider the relationship between solutions of the forward and backward equations for a more general class of processes and give an account for the mathematical source of the symmetry underlying the duality when drift and diffusion are both constant.  We end with some concluding remarks on the limitations of extending the duality to anisotropic diffusions and the possibility of extending the duality to diffusions under a broader class of drift potentials.


\section{Prior results}


\subsection{The duality in $d=1$} 

In one dimension, we have the well-studied scenario:  a particle with initial position $x_0>0$ diffuses under the presence of a constant velocity field $\sigma$, and we ask how long we expect it to take to hit the origin, $x=0$.  

To find the mean first-passage time in this context, we may directly compute the first-passage distribution by measuring the outward flux of particles (trajectories) at the target $0$.  In other words, we explicitly solve the Fokker-Planck equation governing the occupation density (also known as transition or concentration density), 
\begin{align*}
& u_t  = (Du)_{xx} - (\sigma u)_x, \\
& u(x,0) = \delta (x-x_0)
\end{align*}
subject to the Dirichlet condition $u(0,t)=0$ which allows us to ignore all trajectories that hit 0 even if they return in the future to the positive domain.  This may be solved in a myriad of ways, but perhaps the most elegant is to use the method of images.  If we imagine an anti-particle with initial position $-x_0$, then an anti-Gaussian will satisfy the equation of motion (albeit subject to the same drift as the regular particle, not an ``anti-drift'').  Properly weighting the anti-Gaussian ensures the boundary condition is satisfied, and 
\begin{align}
u(x,t) = & \frac{1}{\sqrt{4\pi D t}} \bigg\{ \exp\left[ \frac{-(x-x_0-\sigma t)^2}{4Dt}\right] \notag \\ 
& \quad -  \exp\left[\frac{\sigma x_0}{D} - \frac{(x+x_0-\sigma t)^2}{4Dt}\right]\bigg\}.
\end{align}
The first-passage distribution $f$ may then be found either by computing the outward flux at 0 or equivalently by the rate of loss of mass:
\[
f(x_0,t) = - u_x(0,t) = -\frac{\partial}{\partial t} \int_0^\infty dx \, u(x,t).
\]
Using subscripts to denote the first-passage distribution given positive or  negative drift, we have
\[
f_\pm(x_0,t) = \frac{x_0}{\sqrt{4\pi Dt^3}} e^{-(x_0 \pm |\sigma| t)^2/4Dt}.
\]
When $\sigma=0$, this exhibits well-known but somewhat surprising behavior:  the particle has probability 1 of hitting $x=0$ (this is known as \textit{recurrence}), but the expected time to do so is infinite.  Unsurprisingly, $\sigma=0$ is a critical point for the dynamical properties of the system.  When $\sigma<0$, the mean first-passage time becomes finite (\textit{strong recurrence}).  On the other hand, when $\sigma>0$, not all trajectories hit 0 (\textit{transience}).  Borrowing intuition from random walks, a particle must experience an increasingly unlikely number of successive successes of flips of the coin in order to overcome the prevailing positive wind.  

We are now in a position to discuss the duality.  To repeat, the claim is that the mean first-passage time (moreover the first-passage distribution itself), when conditioned on actually hitting 0, is independent of the sign of $\sigma$.  Of course, conditioning is only required in the case of positive drift where there exists transient behavior -- trajectories that never hit the origin.  In this case, the probability $H_+$ of hitting 0, as a function of the initial position $x_0$, is given by the zeroth moment of the first-passage distribution, 
\[
H_+(x_0) = \int_0^\infty dt \, f_+(x_0,t) = e^{-\sigma x_0/D}.
\]
The first-passage distribution conditioned on actually hitting 0 is then $f_+/H_+$ and it follows that 
\begin{align}
\frac{f_+(x_0,t)}{H^+(x_0)} 
& = \frac{1}{e^{-\sigma x_0/D}} \frac{x_0}{\sqrt{4\pi Dt^3}} e^{-(x_0 + \sigma t)^2/4Dt} \notag\\
& = \frac{x_0}{\sqrt{4\pi Dt^3}} e^{(-x_0^2 + 4\sigma x_0t -2\sigma x_0t -\sigma^2t^2 )/4Dt} \notag\\
& =  \frac{x_0}{\sqrt{4\pi Dt^3}} e^{-(x_0-\sigma t)^2/4Dt}\notag\\
& = f_- (x_0,t). \phantom{\frac{1}{1}} \label{1Dduality}
\end{align}
This is the duality in one dimension.  If we write $T_n$ for the moments of the first-passage distribution then 
\[
T_1(x_0) = \int_0^\infty dt \, f_-(x_0,t) t = \frac{x_0}{|\sigma|}.
\]
In essence, the stronger the wind is pushing the particle toward 0, the faster one expects it to hit 0, and the stronger the wind is pushing the particle away from 0, the trajectories that do hit 0 must do so increasingly quickly, on average, as if they are experiencing twice the negative wind so as to not get pushed too far away and become transient. 




\subsection{The duality in $d=2$}

The two dimensional analog of interest is a diffusing particle subject to a radially symmetric potential velocity field of the form $\vec\sigma = \frac{1}{r} \hat r$ with a target set of a circle centered at the origin.\footnote{One could imagine other higher dimensional analogs, for example diffusing toward a hyperplane in the presence of some constant linear drift.  This however would reduce to the one dimensional case by projecting the diffusing behavior and the constant drift to the direction orthogonal to the target hyperplane.}  By the radial symmetry of the problem, angular movement of the particle bears less importance and  thus our focus is on the radial coordinate.  We denote the initial position $r$ and let $a<r$ be the radius of the target circle.  For the purposes of exposition, we will focus only on the duality for the mean first-passage time in two dimensions, and omit the details for the first-passage distributions, which may be found in~\cite{duality}.  

For more complex problems, solving directly for the occupation density (let alone the first-passage distribution) may not be analytically tractable.\footnote{To find an analytic form of the most general form of the first-passage distribution, for a one-dimensional Ornstein-Uhlenbeck process with a target set $a\neq 0$ remains an open question \refp{}.} There is however no shortage of methods of computing moments of the first-passage distribution, whether it be via Siegert's recursion relation~\cite{Siegert}, Dynkin's formula (for the mean first-passage time) ~\cite{Oksendal}, or deriving various ordinary differential equations for the moments from either the backward equation or biased random walk (known as the Laplacian formalism~\cite{BookRedner}).  Working from the random walk is particularly useful as it also allows us to also consider the conditioned hitting time.  

Let $H(r)$ be the probability of hitting the target circle given initial position $r$, and let $T(r)$ be the (unconditioned) expected first-passage time.  These are of course the zeroth and first moments of the first-passage distribution. 

Starting from a biased random walk, both $H$ and $T$ may be characterized as weighted averages of the hitting probability and first-passage times of accessible neighboring sites to $r$,
\begin{align}
H(r) & = \sum_{r'} p(r\to r') H(r') \label{eqn:H}\\ 
T(r) & = \sum_{r'} p(r\to r') (T(r') + \delta t ). \label{eqn:T}
\end{align}
Here, $p$ is the transition probability to go from $r$ to $r'$ in time $\delta t$.  The additional term of $\delta t$ in\refp{eqn:T} accounts for the time it would take to transition between states $r$ and $r'$.  In the continuum limit,\refp{eqn:H} and\refp{eqn:T} become a pair of homogeneous and inhomogeneous elliptic equations whose operator is that of the backward Kolmogorov equation: 
\begin{align}
DH'' + \left(\frac{D+\sigma}{r}\right) H' & = 0 \label{eqn:Hode}\\
DT'' + \left(\frac{D+\sigma}{r} \right) T' & = -1 \label{eqn:Tode}
\end{align}
where the constants $D$ is the diffusion coefficient and $\sigma$ the strength of the drift.

Solving\refp{eqn:Hode} subject to the conditions $H(a) = 1$ and $\lim_{r\to\infty}H(r) = 0$, it follows that 
\begin{align}
H_-(r) & = 1\\
H_+(r) & = \left(\frac{r}{a}\right)^{-\sigma/D} \label{Hplus}
\end{align}
where subscript denotes the sign of the radial drift.  

The convenience of starting from the random walk is that it also allows us to derive a differential equation for the \textit{conditional} expected first-passage time, which we denote $\tilde T$.  To write this formally, let $\mathcal{P}_r$ be the path distribution given initial position $r$ and let $T^q$ be the first-passage time of a particular path $q$.  Then 
\[
\tilde T(r) = \frac{\sum_{q} \mathcal{P}_r(q) T^q(r) }{\sum_{q} \mathcal{P}_r(q) } = \frac{\sum_{q} \mathcal{P}_r(q) T^q(r) }{H(r)},
\]
where the sums are over all paths $q$ that start at $r$ and eventually hit the target set.  The approach taken to derive\refp{eqn:Tode} from\refp{eqn:T} may be then generalized to yield a new elliptic equation governing the conditional mean first-passage time,
\begin{equation}
D(H\tilde T)'' + \left(\frac{D+\sigma}{r} \right) (H\tilde T)' = -H,\label{eqn:Tc}
\end{equation}
where $H$ is known.  In the case of negative drift, $H_-(r)=1$ and the conditional equation reduces to the unconditional one\refp{eqn:Tode}.  

The statement of the duality for the expected first-passage time in two dimensions is then 
\begin{equation}
\tilde T_+ = T_-.\label{eqn:2Dduality}
\end{equation}
To see this explicitly we may manually solve the negative drift version of\refp{eqn:Tode} and the positive drift version of\refp{eqn:Tc} with $H_+$ given by\refp{Hplus}.  In both cases the boundary condition at the target set is $T(a) = 0$, but some consideration of the physical behavior of the system is required \cite{duality} to reduce the one-parameter family of solutions to find
\[
\tilde T_+(r) = T_-(r) = 
\begin{cases} \dfrac{D(r^2-a^2)}{2(|\sigma|-2D)} & |\sigma| > 2D \\ \infty & |\sigma|\leq 2D.
\end{cases}
\]
This is the statement of the duality for the expected first-passage time in two dimensions.  The two dimensional case therefore exhibits marginally more complex behavior than that of the one-dimensional system; there is now a recurrent but not strongly recurrent regime when drift is only slightly negative.  Heuristically, an extra dimension allows for enough additional `room' for a diffusing particle to wander in, and a mild negative drift isn't enough to ensure a finite expected hitting time (\textit{strong recurrence}).   

An alternative approach, rather than direct computation, is to recast\refp{eqn:2Dduality} as a statement about solutions to the governing ODEs:  if $H_+$ solves\refp{eqn:Hode} and $T_-$ solves the negative drift version of\refp{eqn:Tode} then $T_-$ (given $H_+$) solves the positive drift version of \refp{eqn:Tc}.  This is straightforward to verify, and is the precisely the nature of the argument that we generalize to demonstrate the duality in higher dimensions, albeit for the first-passage distributions and not just the first-passage time.


\section{Extending the duality to $d>2$}

The higher dimensional analog of the two dimensional scenario is a diffusing particle with a radially symmetric drift of the form $\frac{v}{\rho^{d-1}} \hat\rho$, with a target set of a $d$-sphere, where $v$ is the strength of the drift and $\rho$ is the radial coordinate.  Via hyperspherical coordinates the equation of motion may be reduced to one of a single spatial variable,
\[
u_t = D\left( u_{\rho\rho} + \frac{d-1}{\rho} u_\rho\right) - \frac{v}{\rho^{d-1}} u_\rho.
\]
As before, we impose a Dirichlet condition on the target set: $u(a,t) = 0$ where $a<r$ is the radius of the hypersphere.  By defining $\sigma = v/D$ we may write this in the mildly simpler form 
\begin{equation}
u_t = \mathcal{L}^*_\rho u:= D\left[ u_{\rho\rho} + \left(\frac{d-1}{\rho} - \frac{\sigma}{\rho^{d-1}}\right) u_\rho\right] \label{eqn:genforward}
\end{equation}
where we have introduced $\mathcal{L}^*$ as the spatial operator governing the process.\footnote{We use a star superscript here to denote the adjoint, reserving $\mathcal{L}$ to be associated with the infinitesimal generator, the spatial operator in the backward equation.} \\
\indent As in the two dimensional case, we make no attempt to solve this generally.  Since we are after a proof of the duality for the first-passage distributions themselves, our goal is prove an equivalence analogous to\refp{1Dduality}.  We could do this directly for the occupation density (see $\mathsection 4$), although a more intuitive approach is to appeal to the survival probability $S$.  Given the radial symmetry of the dynamics, the survival probability, defined to be the probability of not yet having hit the target set at time $t$, may be expressed as a function of the initial radial distance $r$ from the origin, i.e.
\[
S(r,t) = \int_a^\infty d\rho\, u(\rho,t|r,0).
\]
Note that $S$ depends on an initial state and not a terminal one, and is therefore not directly governed by the forward equation.  Due to the time-homogeneity of the process however, $u(\rho,t|r,0) = u(\rho,0|r,-t)$ and therefore 
\[
\partial_t S = \int_a^\infty d\rho \, \partial_t u(\rho,0| r,-t) = \int_a^\infty d\rho \, \mathcal{L}_{r} u = \mathcal{L}_{r} S.
\]
In other words, the survival probability is governed by the backward equation, where $\mathcal{L}$ is given by 
\begin{equation}
\mathcal{L}_r g := D \left[g_{rr} + \left( \frac{d-1}{r} + \frac{\sigma}{r^{d-1}}\right) g_r\right]. \label{Loperator}
\end{equation}
Note that the usual procedure of finding the adjoint is most easily computed in Euclidean coordinates where the forward equation in general is given by 
\[
u_t = \Delta (Du) - \nabla \cdot (\vec\sigma u).
\]
This ensures the appropriate sign changes; only the sign of the drift changes and \textit{not} the sign of the component of diffusion contributing to the first order term (which may occur from a naive attempt in spherical coordinates).  The adjoint \textit{is} computable directly from\refp{Loperator} however integration against a test function must include the appropriate Jacobian to account for the coordinate change.  

Since the first-passage distribution $f$ is simply $f = \partial_t S$, if we can prove a statement of the form $S_- = S_+/H_+$ then $f$ will inherit the duality.  Let us now distinguish between positive and negative drift backward equations by $\mathcal{L}^\pm$ for the spatial operator.

\begin{theorem}
Let $D>0$ and $\sigma>0$ be fixed.  Suppose that $S_+(r,t)$ is a solution to the backward equation in $d>2$ dimensions,
\begin{align}
& \partial_t S_+ = \mathcal{L}^+_{r} S_+ := D\left[\partial_{rr}S_+ + \left(\frac{d-1}{r} + \frac{|\sigma|}{r^{d-1}} \right) \partial_r S_+\right]\label{eqn:genbackward}\notag\\
& S_+(a) = 1, \quad\quad  \lim_{r\to\infty} S_+(r) =0
\end{align}
and $H_+(r)$ is the probability of hitting the $d$-dimensional hypersphere of radius $a$ from an initial position $r$.  Then $S_+/H_+$ solves the negative drift backward equation, i.e. 
\begin{equation}
\partial_t \frac{S_+}{H_+} = \mathcal{L}^-_{r} \frac{S_+}{H_+}.\label{eqn:theorem1}
\end{equation}
Similarly, if $S_-$ is a solution to the negative drift backward equation, then $S_- H^+$ solves the positive drift backward equation.  
\end{theorem}

In other words, there is a one-to-one correspondence between the survival probability in the appropriate cases of positive and negative drift (so long as the magnitude is fixed).  For notational simplicity, we prove the second statement, that $S_-H_+$ solves the positive drift backward equation, from which the main statement of the duality follows.  

\begin{proof}
As in the two dimensional case, the probability $H_+$ of hitting the $d$-sphere is found by solving the ordinary differential equation 
\[
0 = \mathcal{L}_{r}H_+ = D\left[H_{+,rr} + \left(\frac{d-1}{r} +\frac{\sigma}{r^{d-1}} \right)H_{+,r}\right]
\]
Solving this subject to the boundary conditions $H_+(a)=1$ and $\lim_{r\to\infty} H_+(r) = 0$ gives us
\begin{equation}
H_+(r) = \exp\left[ \frac{\sigma}{d-2} \left(\frac{1}{r^{d-2}} -\frac{1}{a^{d-2}}\right)\right]. \label{hplus}
\end{equation}
Before proceeding, note that
\begin{equation}
H_+' = -\frac{\sigma}{r^{d-1}}H_+.
\label{hplusd}
\end{equation}
Now let us suppose that $S_-$ solves the negative drift backward equation, with drift coefficient $-|\sigma|$.  Then, dividing through by $D$ and temporarily using primes on $S$ to denote differentiation with respect to $r$,  
\begin{align*}
\frac{1}{D}\mathcal{L}_r^+ (S_-H_+)
& = (S_-H_+)'' +  \left(\frac{d-1}{r} +\frac{|\sigma|}{r^{d-1}} \right) (S_-H_+)' \\
& = S_-'' H_+ + 2S_-'H_+' +  \left(\frac{d-1}{r} +\frac{|\sigma|}{r^{d-1}} \right) S_-'H_+ \\
& \quad + S_- \left[ H_+'' +  \left(\frac{d-1}{r} +\frac{|\sigma|}{r^{d-1}} \right) H_+' \right] \\
& = S_-''H_+ + 2S_-'H_+' + \left(\frac{d-1}{r} +\frac{|\sigma|}{r^{d-1}} \right) S_-'H_+ \\
& = \left[S_-'' + \left(\frac{d-1}{r} -\frac{|\sigma|}{r^{d-1}} \right) S_-'\right]H_+  \tag{by \refp{hplusd}} \\
& = \frac{1}{D}(\mathcal{L}^-_r S_-)H_+. \phantom{\int}
\end{align*}
Therefore, 
\[
\mathcal{L}_r^+ (S_-H_+) = (\mathcal{L}^-_r S_-)H_+ = (\partial_t S_-)H_+ = \partial_t (S_-H_+),
\]
and $S_-H_+$ solves the positive drift equation.  
\end{proof}

This demonstrates that we may convert between solutions of the positive and negative drift equations by way of the hitting probability $H_+$, and that in general, 
\[
\frac{S_+}{H_+} = S_-.
\]
The proof of Theorem 1 can be identically recast to work in the setting $d=2$, albeit for a different analytic form of $H_+$, as\refp{hplus} is not defined for $d=2$.  
\begin{theorem}
The first-passage distribution for a particle diffusing in $d$-dimensions to a hypersphere in the presence of a constant radial drift, conditioned on actually hitting the sphere, is independent of the sign of the drift.  
\end{theorem}

\begin{proof} 
This follows immediately from Theorem 1 and the fact that $f_\pm = \partial_t S_\pm$.  In other words, 
\begin{equation}
\frac{f_+(r,t)}{H_+(r)} = f_-(r,t).\label{genduality}
\end{equation}
\end{proof}

The existence of the duality in any positive integer dimension begs the question as to the mathematical source of the symmetry.  The key observation to make is that in going from\refp{eqn:genforward} to\refp{eqn:genbackward}, while we formally compute the adjoint in Euclidean coordinates before converting back to hyperspherical coordinates, the \textit{effect} of computing the adjoint was only to change the sign of the drift magnitude $\sigma$.  We now illustrate how this allows for the direct relation\refp{genduality} between the forward and backward hitting distributions.


\section{The adjoint operator and the underlying symmetry}

Consider the forward and backward equations governing the transition density $u(x,t|y,s)$ of a one-dimensional time-\\ homogeneous process (where the drift and diffusion coefficients may now have spatial dependence):
\begin{align*}
u_s & = -\mathcal{L}_y u := -Du_{yy} - \sigma u_y \tag{B} \\
u_t & = \phantom{-} \mathcal{L}_x^* u \phantom{:}= (Du)_{xx} - (\sigma u)_x \tag{F}
\end{align*}
As was stated for the survival probability, by the time-homogeneity of the process the backward equation may be solved `forward' in time.  If we define $v(\,\cdot \, , t) = u(\,\cdot \,, t-s)$ then (B) becomes
\begin{equation}
v_t = \mathcal{L}_y v = Dv_{yy} + \sigma v_y \tag{B'}
\end{equation}

A natural question to ask is how solutions to (F) and solutions to (B') are related. \\

\noindent \textbf{Proposition.}
\textit{If $v$ solves} (B') \textit{and $u_0$ solves} $\mathcal{L}^*_x u = 0$ \textit{then} $u(x,t)=v(x,t) u_0(x)$ \textit{solves} (F).  

\begin{proof}
We first find $u_0$ directly.  Setting $\mathcal L^*_xu=0$, integrating once, and rearranging yields
\begin{equation}
u_0(x) = \exp\left[ \int^x \frac{\sigma - D'}{D}\right].
\end{equation}
It follows that 
\begin{align*}
\mathcal{L}^*_x u & = (Dvu_0)_{xx} - (\sigma vu_0)_{x}\\
& = Dv_{xx}u_0 + 2 (Du_0)_x v_x + (Du_0)_{xx}v - (\sigma u_0)_xv - \sigma v_x u_0.
\end{align*}
The third and fourth terms cancel by assumption:  $u_0$ solves $\mathcal L^*_x u_0=0$.  Since $D(u_0)_x = (\sigma - D_x)u_0$, it follows that $(Du_0)_x = \sigma u_0$ and thus
\begin{equation}
\mathcal{L}^*_x u  = (Dv_{xx} - \sigma v_x)u_0  = v_t u_0 = (vu_0)_t = \partial_t u.
\end{equation}
\end{proof}

The source of the symmetry in the first-passage duality\refp{1Dduality} and\refp{genduality} is the fact that for constant diffusion and constant drift, the spatial operator found by reversing the sign of the drift is the same as that found by taking the adjoint.  By way of Prop.\,1, the solutions to the positive and negative drift equations therefore differ only by a factor of $u_0$.  Note that the proposition holds irrespective of boundary conditions, however if we impose $u_0(a)=1$ and $\lim_{x\to\infty}u_0(x)=0$ then $u_0=H$ and is the exact factor needed for conditioning the first-passage distributions.  This also makes it clear that we could have proved Theorems 1 and 2 from the transition density and a general form of $u_0$, but we opted to use the conditioned survival probability as it is somewhat easier to intuit than the object $u/u_0$.


\section{Concluding remarks}

We have extended Krapivsky and Redner's simple duality for first-passage time distributions to dimensions greater than 2.  In order to do so, we recognized that the mathematical source of the symmetry underpinning the duality is the fact that computing the adjoint of the spatial operator is equivalent to reversing the sign of the drift.  Not only does this shed light on the duality, it accommodates proving it without solving the governing dynamical equations directly, a task that becomes formidable in higher dimensions.  

There are two obvious directions to explore for future work.  The first is to consider non-spherical target sets in dimensions greater than one.  Given that dividing by $u_0$ always allows us to convert forward equation solutions to backward ones, irrespective of boundary conditions, it may be possible that the duality, or some approximation of it, may extend to a larger class of target sets.  A reasonable place to start would be convex target sets whose boundaries are smooth and exhibit some radial symmetry, or perhaps star-shaped neighborhoods of 0 (convex with respect to the origin).  Preliminary numerical experiments suggest that while the duality should not hold true, it may be possible to bound the extent to which the inward and conditioned outward hitting times disagree based on the geometry of the target set.   

The other and possibly more interesting direction to consider would be the extent to which the duality holds for other time-homogeneous processes.  In one dimension this is quickly understood; the duality does not extend to anisotropic diffusion or non-constant drift.  For $\mathcal L_+ f:= (Df)'' + (\sigma f)'$, then 
\begin{align*}
0 & = \mathcal L_- f- \mathcal L_+^*f \\
& =  [ (Df)'' - (\sigma f)' ]-D f'' - \sigma f' \\
& = 2D'f' + D''f - \sigma' f. 
\end{align*} 
For an arbitrary $f$, $2D'f'=0$ implies $D$ be constant and in turn $D'' f - \sigma' f=0$ implies that $\sigma$ be constant. 

In higher dimensions, the picture is somewhat less clear.  In the radial potential drift scenario, the symmetry of the problem allowed us to reduce it to one spatial variable and it was in this characterization that we explored the duality.  If, however, we start from Euclidean coordinates and again set $(\mathcal L_- -\mathcal L_+^*) f =0$, for $f=f(\vec x)$ then 
\begin{align*}
0 & =  \mathcal L_- f - \mathcal L_+^*f \\
& = \Delta (Df) - \nabla \cdot (\vec\sigma f) - (D\Delta f - \vec\sigma \cdot (\nabla f)) \\
& = 2\nabla D \cdot \nabla f + (\Delta D - \nabla \cdot \vec\sigma) f
\end{align*}
As before, given no restrictions on $f$, the first term implies $D$ be constant, but now the second term gives a generalized condition on the drift:  $\vec\sigma$ must be a divergence-free velocity field.  

Extending the duality to some subset of divergence-free drifts is clearing an enticing goal and could be a substantial generalization.  It is immediately clear that the inclusion of a purely tangential component to the drift, in the scenario studied in this paper, should not break the duality.  In two dimensions, for example, if we add some spiraling tendency such as $\langle -y, x\rangle$ to the drift we have studied thus far, the problem still projects cleanly back to the radial component.  The more intriguing processes to then consider are those that are purely radial but non-constant with respect to $\theta$.  These may be written generally as 
\[
\sigma_f(r,\theta) = \frac{f(\theta)}{r} \hat r
\] 
where $f$ is a smooth, periodic, strictly positive function with $f(0) = f(2\pi)$.  Remaining strictly positive guarantees that reversing the sign of the drift sufficiently alters the dynamics, between recurrence and transience, a key component in the set up of the duality.  

Preliminary numerical results suggest that non-constant functions $f$ break the duality, but it may be possible to reasonably estimate the inward or conditioned outward expected hitting time, given an initial position $\vec r$ based on the local average of $f$.  Speculation as to what happens when allowing for tangential drift in addition to a non-constant radial component is beyond our current level of intuition of the problem.  

A deeper exploration of these possible generalizations presents an exciting project and will be the possible focus of future work.


\bibliographystyle{plain}
\bibliography{FPdualityBib}

\end{document}